\documentclass[11pt,leqno]{amsart}

\usepackage{color}

\newtheorem{theorem}{Theorem}
\newtheorem{lemma}[theorem]{Lemma}
\newtheorem{proposition}[theorem]{Proposition}
\newtheorem{cor}[theorem]{Corollary}

\newcommand{\ol} {\overline}
\newcommand{\Si} {\Sigma}

\newcommand{\de}{\delta}
\newcommand{\mc}{multilinear commutator word}
\author{Eloisa Detomi}
\address{Dipartimento di Matematica, Universit\`a di Padova, 
 Via Trieste 63, 35121 Padova , Italy}
\email{detomi@math.unipd.it}
\author{Marta Morigi}
\address{Dipartimento di Matematica, Universit\`a di Bologna,
Piazza di Porta San Donato 5, 40126 Bologna, Italy}
\email{marta.morigi@unibo.it}
\author{Pavel Shumyatsky}
\address{Department of Mathematics, University of Brasilia,
Brasilia-DF, 70910-900 Brazil}
\email{pavel@unb.br}

\title{On countable coverings of word values in profinite groups}

\begin{document}

\begin{abstract} We prove the following results. Let $w$ be a multilinear commutator word. If $G$ is a profinite group in which all $w$-values are contained in a union of countably many periodic subgroups, then the verbal subgroup $w(G)$ is locally finite. If $G$ is a profinite group in which all $w$-values are contained in a union of countably many subgroups of finite rank, then the verbal subgroup $w(G)$ has finite rank as well. As a by-product of the techniques developed in the paper we also prove that if $G$ is a virtually soluble profinite group in which all $w$-values have finite order, then $w(G)$ is locally finite and has finite exponent.
\end{abstract}
\maketitle

\section{Introduction}

Let $w$ be a group-word in $n$ variables, and let $G$ be a group. The verbal subgroup $w(G)$ of $G$
determined by the word $w$ is the subgroup generated by the set consisting of all values $w(g_1,\ldots,g_n)$, where $g_1,\ldots,g_n$ are elements of $G$. In the present paper we deal with the so called multilinear commutators (otherwise known under the name of outer commutator words). These are words which are obtained by nesting commutators, but using always different variables. Thus the word $[[x_1,x_2],[x_3,x_4,x_5],x_6]$ is a multilinear commutator while the Engel word $[x_1,x_2,x_2,x_2]$ is not. An important family of multilinear commutators are the simple commutators $\gamma_k$, given by $$\gamma_1=x_1, \qquad \gamma_k=[\gamma_{k-1},x_k]=[x_1,\ldots,x_k].$$
The corresponding verbal subgroups $\gamma_k(G)$ are the terms of the lower central series of $G$. Another distinguished sequence of multilinear commutator words is formed by the derived words $\delta_k$, on $2^k$ variables, which are defined recursively by
$$\delta_0=x_1,\qquad \delta_k=[\delta_{k-1}(x_1,\ldots,x_{2^{k-1}}),\delta_{k-1}(x_{2^{k-1}+1},\ldots,x_{2^k})].$$
Of course $\delta_k(G)=G^{(k)}$, the $k$th derived subgroup of $G$. 

In recent years the situation where the set of $w$-values in $G$ is covered by finitely many subgroups has been given
certain attention.

In this direction we mention the following result that was obtained in \cite{RS}. Let $w$ be either the
lower central word $\gamma_k$ or the derived word $\delta_k$. Suppose that $G$ is a group in which  all
$w$-values are contained in a union of finitely many Chernikov subgroups.  Then $w(G)$ is Chernikov.  Recall
that a group is Chernikov if and only if it is a finite extension of a direct sum of finitely many Pr\"ufer groups
$C_{p^\infty}$.

Another result of this nature was established in \cite{FS}: If $G$ is a group in which  all commutators are
contained in a union of finitely many cyclic subgroups, then $G'$ is either cyclic or finite. Later  Cutulo
and  Nicotera  showed that if $G$ is a group in which  all $\gamma_k$-values are contained in a union of 
finitely many cyclic subgroups, then $\gamma_k(G)$ is finite-by-cyclic. They also showed that $\gamma_k(G)$ can
be neither cyclic nor finite \cite{CN}.

The paper \cite{AS} deals with profinite groups in which all $w$-values are contained in a union of finitely
many subgroups with certain prescribed properties.

A profinite group is a topological group that is isomorphic
to an inverse limit of finite groups.  The textbooks \cite{ribes-zal}  and \cite{book:wilson} provide a good introduction to
the theory of profinite groups.  In the context of profinite groups all the usual concepts of group theory are
interpreted topologically. In particular, in a profinite group the verbal subgroup corresponding to the word $w$
is the {\it closed} subgroup generated by all $w$-values. More generally, throughout this paper by a subgroup of a profinite group we 
always mean a closed subgroup and by a quotient we mean a quotient over a normal closed subgroup. 
The following theorem was proved in \cite{AS}.
\bigskip
  
{\it Let $w$ be a multilinear commutator word and $G$ a profinite group that has finitely many periodic subgroups whose union contains all $w$-values in $G$. Then $w(G)$ is locally finite.}
\bigskip

A group is periodic (torsion) if every element of the group has finite order and  a group 
is called locally finite if each of its finitely generated subgroups is finite. Periodic profinite groups have received a good deal of attention in the past. In particular, using Wilson's reduction theorem  \cite{wilson-torsion}, 
Zelmanov has been able to prove local finiteness of periodic compact
groups \cite{z}.  Earlier Herfort showed that there exist only finitely many primes dividing the orders of elements of a periodic profinite group \cite{he}.  It is a long-standing problem whether any periodic profinite group has finite exponent. Recall that a group $G$ has exponent $e$ if $x^e=1$ for all $x\in G$ and $e$ is the least positive integer with that property.
Another result obtained in \cite{AS} is as follows.
\bigskip
  
{\it Let $w$ be a multilinear commutator word and $G$ a profinite group that has finitely many subgroups of finite rank
whose union contains all $w$-values in $G$. Then $w(G)$ has finite rank.}
\bigskip

Recall that a profinite group is said to be of finite rank $r$ if each subgroup of $G$ can be generated by at most $r$ elements.

In the present paper we study profinite groups in which all $w$-values are covered by countably many subgroups.
In particular, we will prove the following theorems.

\begin{theorem}\label{theoa} Let $w$ be a multilinear commutator word and $G$ a profinite group that has countably
many periodic subgroups whose union contains all $w$-values in $G$. Then $w(G)$ is locally finite.
\end{theorem}
\begin{theorem}\label{theob} Let $w$ be a multilinear commutator word and $G$ a profinite group that has countably
many subgroups of finite rank whose union contains all $w$-values in $G$. Then $w(G)$ has finite rank.
\end{theorem}
In the particular case where $w=[x,y]$ both above theorems were earlier obtained in \cite{as}. 

It is easy to see that if $w$ is a multilinear commutator word and $G$ is an abstract soluble group in which all $w$-values
have finite order, then $w(G)$ is locally finite (see for example \cite[Lemma 4.2]{S2}). As a by-product of the techniques developed in the present paper we obtain a profinite version of this fact.

\begin{theorem}\label{theoc} Let $w$ be a multilinear commutator word and $G$ a virtually soluble profinite group 
in which all $w$-values have finite order. Then $w(G)$ is locally finite and has finite exponent.
\end{theorem}

It seems an interesting open problem whether the assumption that $G$ is virtually soluble can be dropped from the
hypothesis of Theorem \ref{theoc}. In other words, let $G$ be a profinite group in which all $w$-values have finite
order. Is $w(G)$ necessarily locally finite? 

A close inspection of the proof of Theorem \ref{theoa} reveals that if $G$ is a profinite group that has countably
many subgroups of finite exponent whose union contains all $w$-values, then $w(G)$ has finite exponent as well. It
seems plausible that if $G$ is a profinite group in which all $w$-values are contained in a union of countably many subgroups of finite exponent dividing $e$, then $w(G)$ has an open subgroup of finite exponent dividing $e$. We were unable to prove the latter statement, though. Similarly, if $G$ is a profinite group in which all $w$-values are contained in a union of countably many subgroups of finite rank at most $r$, we conjecture that $w(G)$ necessarily has an open subgroup of rank at most $r$.

In the next section we develop some technical tools which will be of crucial importance in the proofs of the main results.  The proofs will be given in Section 3.

\section{Combinatorics of commutators}

Throughout this section, $G$ will be an abstract group. Some results given here were obtained in an unpublished work by the second author and G.\ Fern\'andez-Alcober.

\begin{lemma}
\label{factor goes to modulus}
Let $w=w(x_1,\ldots,x_n)$ be a \mc\ on $n$ variables, and let 
$k\in\{1,\ldots,n\}$. Let $M, A_1, \ldots, A_n$ be normal subgroups of a group $G$ such that 
 $w(A_1,\ldots,A_{k-1},M,A_{k+1},\ldots,A_n)=1$
 and choose elements  $a_i\in  A_i$ for $i=1,\ldots,n$, and  $m \in M$. 
Then
\[
w(a_1,\ldots,a_{k-1},a_km_k,a_{k+1},\ldots,a_n)=
w(a_1,\ldots,a_{k-1},a_k,a_{k+1},\ldots,a_n).
\]
\end{lemma}

\begin{proof}
If $n=1$, the result is obvious. 
Assume that $n>1$ and write $w(x_1,\dots,x_n)=[\varphi(x_1,\dots,x_r),\psi(x_{r+1},\dots,x_n)
]$ where 
  $\varphi$ and $\psi$ are \mc s.
  
Suppose that $k\le r$. 
For short, let 
\[u=\varphi(a_1,\ldots,a_{k-1},a_k,a_{k+1},\ldots,a_r) \text{ and } 
v=\psi(a_{r+1},\dots,a_n).\]
 Since $ r<n$, by induction we have 
\[
\varphi(a_1,\ldots,a_{k-1},a_km_k,a_{k+1},\ldots,a_r)=
zu,
\]
where $z\in \varphi(A_1,\ldots,A_{k-1},M_k,A_{k+1},\ldots,A_r)$.
Then 
\[
w(a_1,\ldots,a_{k-1},a_km,a_{k+1},\ldots,a_n)=
\]
\[
[\varphi(a_1,\ldots,a_{k-1},a_km,a_{k+1},\ldots,a_r),\psi(a_{r+1},\dots,a_n)]=
[zu,v]=
[z,v]^u[u,v].
\]
 Taking into account that  $[z,v]^{u}
 \in w(A_1,\ldots,A_{k-1},M,A_{k+1},\ldots,A_n)=1$, we deduce that 
\[w(a_1,\ldots,a_{k-1},a_km,a_{k+1},\ldots,a_n)=w(a_1,\ldots,a_{k-1},a_k,a_{k+1},\ldots,a_n)\] 
and the case $k\le r$ is proved.

If $ k > r$,  the result follows from the previous case and the fact that $[y_1,y_2]=
[y_2,y_1]^{-1}$ for every $y_1,y_2\in G$.
\end{proof}

Let $n\geq 1$. We denote by $I$ the set of all $n$-tuples $(i_1,\dots,i_n)$, where all entries $i_k$ are non-negative integers. We will view $I$ as a partially ordered set with the partial order given by the rule that $$(i_1,\dots,i_n)\leq(j_1,\dots,j_n)$$ if and only if $i_1\leq j_1,\dots,i_n\leq j_n$.

Given a  multilinear commutator word  $w=w(x_1,\dots,x_n)$ and  
 $\mathbf{i}=(i_1,\ldots,i_n) \in I$, we write
\[
w(\mathbf{i})=w(G^{(i_1)},\ldots,G^{(i_n)})
\]
 for  the subgroup generated by the $w$-values $w(g_1,\dots,g_n)$ with  $g_{j} \in G^{(i_j)}$. 
Further, set 
\[
w(\mathbf{i^+})=\prod w(\mathbf{j} ),
\]
 where the product is taken over all $\mathbf{j} \in I $ such that $\mathbf{j}>\mathbf{i}$.

\begin{lemma}
 \label{conj} Let $w=w(x_1,\ldots,x_n)$ be a \mc\ and $\mathbf{i} \in I$. For every $j=1,\ldots,n$ choose $a_j\in G^{(i_j)}$. 
 Let $x\in G^{(s)}$ for some integer $s$.
  Then
\[
 w(a_1,\dots,a_n)^x\equiv w(\overline{a}_1,\dots,\overline{a}_n) 
\pmod{ w(\mathbf{i^+})},
\]
where $\overline{a}_k=a_k$ if $ i_k \le s $ and $\overline{a}_k=a_k^x$ otherwise.
\end{lemma}
\begin{proof} We have   $ w(a_1,\dots,a_n)^x= w(a_1^x,\dots,a_n^x)$.
Assume that $ i_k \le s$. Then $a_k^x=a_k[a_k,x]$ and $[a_k,x]\in [G^{(i_k)},G^{(i_k)}]\le G^{(i_k+1)}$. Now an application of 
 Lemma
\ref{factor goes to modulus} with  $m_k=[a_k,x]$ completes the proof.
\end{proof}

\begin{cor}\label{cor:ab}
 Let $w=w(x_1,\ldots,x_n)$ be a \mc\  and $\mathbf{i} \in I$. If $ w(\mathbf{i^+})=1$, then $ w(\mathbf{i})$ 
 is abelian.  
\end{cor}
\begin{proof} 
Let $a_j, b_j \in G^{(i_j)}$ for every  $j=1,\ldots,n$. Since  $w(\mathbf{i^+})=1$ and 
$w(b_1,\dots,b_n)\in G^{(s)}$ where $s= \max \{i_1, \ldots , i_n\}$, by Lemma \ref{conj} it follows that 
$[  w(a_1,\dots,a_n), w(b_1,\dots,b_n)]=1$, and the result  follows. 
\end{proof}

 \begin{proposition}
\label{skew congruence}
 Let $w=w(x_1,\ldots,x_n)$ be a \mc\ on  $n$ variables and  
let $\mathbf{i} \in I$. 
For every $j=1,\ldots,n$, consider an element $a_j\in G^{(i_j)}$, and
for a fixed value $k$, let $b_{k}\in G^{(i_k)}$.
If $ w(\mathbf{i^+})=1$, then
\begin{eqnarray*}
& w(a_1,\ldots,a_{k-1},b_{k} a_{k},a_{k+1},\ldots,a_{n}) 
= \\
& w(\overline a_1,\ldots,\overline a_{ k-1},b_{ k},\overline a_{{ k}+1},\ldots,\overline a_{n})
w (a_1,\ldots,a_{{ k}-1},a_{ k},a_{{ k}+1},\ldots,a_{n}),
\end{eqnarray*}
where 
 $\overline a_j$ is a conjugate of $a_j$ and moreover 
$\overline a_j=a_j$  if $i_j \leq i_k$. 
\end{proposition}

\begin{proof}
 We argue by induction on  $n$.
If $n=1$,  then the result is obvious so we assume that $n\ge 2$.

For short, we will write $\ol a$ to indicate an element which is conjugate to $a$  
 in $G$.
We can write $w(a_1,\dots,a_n)=[\varphi(a_1,\dots,a_r),\psi(a_{r+1},\dots,a_n)]$ where 
  $\varphi$ and $\psi$ are \mc s.  
 Moreover, 
  we write $\mathbf{i}=(\mathbf{i}_1,\mathbf{i}_2)$ where 
 $\mathbf{i}_1=(i_1,\ldots,i_r)$  and $\mathbf{i}_2=(i_{r+1},\ldots,i_n)$. 

Assume first that $k \le r$. Let 
\[
u =\varphi(a_1,\dots, b_{ k} a_{ k},\ldots,a_r),
\]
and
\[
v =\psi(a_{r+1},\dots,a_n).
\]
As $r <n$, we can apply the inductive hypothesis to conclude  that 
there exists an element $z \in \varphi(\mathbf{i}_1^+)$ such that 
\[
u = u_1u_2 z,
\]
with
\begin{eqnarray*}
u_1 &=& \varphi(\overline a_1,\ldots,\overline a_{ k-1},b_{ k},\overline a_{ k+1},\ldots,\overline a_{r}) \\
u_2 &=&\varphi(a_1,\ldots,a_{ k-1},a_{ k},a_{ k+1},\ldots,a_{r}),
\end{eqnarray*}
where 
$\overline a_j=a_j$   if $i_j \le i_k$,  for  $j=1,\dots,r$. 

Since  $   [ \varphi(\mathbf{i}_1^+), \psi(\mathbf{i}_2)] \le  w(\mathbf{i}^+)=1$  and  $z \in \varphi(\mathbf{i}_1^+)$,  
 it follows from Lemma \ref{factor goes to modulus} applied to the word $[x_1, x_2]$ that
\begin{equation}\label{firstcase}
[u,v]=[u_1u_2, v].
\end{equation}

Observe that 
\begin{equation*}
[u_1u_2,v] = [u_1,v]^{u_2} [u_2,v] = [u_1^{u_2},v^{u_2}] [u_2,v]. 
\end{equation*}
 By Corollary \ref{cor:ab} 
 \[u_1^{u_2} \equiv u_1  \pmod{ \varphi(\mathbf{i}_1^+)},\]
 and by Lemma \ref{conj}, since $u_2 \in G^{(i_k)}$,  
 \[v^{u_2} \equiv \psi(\overline a_{r+1},\dots,\overline a_n) \pmod{ \psi(\mathbf{i}_2^+)}\]
where 
  $\overline a_j=a_j$   
   if  $ i_j \le i_k$.
Therefore, as $   [ \varphi(\mathbf{i}_1^+), \psi(\mathbf{i}_2)]  [ \varphi(\mathbf{i}_1), \psi(\mathbf{i}_2^+)] \le  w(\mathbf{i}^+)=1$, we have 
\[
[u_1^{u_2},v^{u_2}]= 
[u_1,v^{u_2}]=\]
\[
[\varphi(\overline a_1,\ldots,\overline a_{ k-1},b_{ k},\overline a_{ k+1},\ldots,\overline a_{r}),
\psi(\overline a_{r+1},\dots,\overline a_n)]
=
\]
\[
w(\overline a_1,\ldots,\overline a_{ k-1},b_{ k},\overline a_{{ k}+1},\ldots,\overline a_{n}),
\]
hence  (\ref{firstcase}) becomes
\[
[u,v]= 
w(\overline a_1,\ldots,\overline a_{ k-1},b_{ k},\overline a_{{ k}+1},\ldots,\overline a_{n})
w (a_1,\ldots,a_{{ k}-1},a_{ k},a_{{ k}+1},\ldots,a_{n}),
\]
where 
 $\overline a_j=a_j$ whenever  $i_j \le i_k$,
and  this completes the proof in the case where $k \le r$. 

The proof of the  case   $k > r$ is similar. 
\end{proof}

\section{Proofs of the main results} 

We will first prove Theorems  \ref{theoa} and \ref{theob} for soluble groups. In this case, the result is stated 
in a more general setting.

Let $\Si$ be a property of profinite groups such that: 
\begin{enumerate}
\item Every finite group is a $\Si$-group;
\item The class of all  $\Si$-groups is closed under taking subgroups, quotients and 
extensions.
\end{enumerate}

For example $\Si$ can be the property of being of finite rank, of finite exponent, or periodic. The next lemma proves Theorems \ref{theoa} and \ref{theob}  in the soluble case.

\begin{lemma}\label{lem:sigma} Suppose that $w$ is a multilinear commutator word. Let $G$ be a soluble profinite group that has countably many subgroups  $G_1,G_2,\ldots$ whose union contains all $w$-values in $G$. If every subgroup $G_u$ has the property $\Si$, then also $w(G)$ has the property $\Si$.
\end{lemma}
\begin{proof} Let $n$ be the number of variables involved in $w$.  We will assume that $w(G)\neq1$. As $G$ is soluble, there exist only finitely many 
${\bf i}\in I$ such that $w({\bf i})\neq1$. The lemma will be proved by induction on the number of such tuples ${\bf i}$. 

Choose ${\bf i}=(i_1,\dots,i_n)\in I$ such that $w({\bf i})\neq1$ while $w({\bf j})=1$ whenever ${\bf i}<{\bf j}$. By Corollary \ref{cor:ab} it follows that  $w({\bf i})$ is abelian. We will now show that $w({\bf i})$ has  the property $\Si$. 

Let us consider the subgroups $Y_s$ obtained as a product of the intersections of the  first $s$ subgroups $G_u$ with  $w({\bf i})$:  
\[Y_s =\prod_{u=1}^s ( G_u \cap w({\bf i})).\] 
 Since $w({\bf i})$ is abelian and $\Si$ is closed  under taking subgroups 
 and extensions, 
 the subgroups $Y_s$ are $\Si$-groups and they cover the set of $w$-values lying in $w({\bf i})$. 
Moreover $Y_s\le Y_t$ if $s\le t$. These are the only properties of the subgroups $Y_s$ that will be used in what follows. 
 
Let $\pi$ be a permutation of the set $\{i_1,\dots,i_n\}$ satisfying the condition that $i_{\pi(1)}\leq i_{\pi(2)}\leq\dots\leq i_{\pi(n)}$. We wish to show that for every $k=0,1,\dots,n$ and every choice of $a_{\pi(1)}\in G^{(i_{\pi(1)})}$, $a_{\pi(2)}\in G^{(i_{\pi(2)})}$, $\dots$, $a_{\pi(k)}\in G^{(i_{\pi(k)})}$ there exists an index $s$ such that every $w$-value $w(x_1,\dots,x_n)$, where $x_j\in G^{(i_j)}$ and $x_j=a_j$ whenever $j\in \{\pi(1),\dots,\pi(k)\}$, is contained in  $Y_s$. This will be shown by induction on $n-k$.
 
If $n-k=0$, there is only one element $w(x_1,\dots,x_n)$ of the required form, namely, $w(a_1,\dots,a_n)$. So we only need to show that the element $w(a_1,\dots,a_n)$ belongs to some of the subgroups $Y_j$, which it true by the hypothesis. Thus, we assume 
that $n-k\geq 1$ and  that the elements $a_{\pi(1)}\in G^{(i_{\pi(1)})}$, $a_{\pi(2)}\in G^{(i_{\pi(2)})}$, $\dots$, $a_{\pi(k)}\in G^{(i_{\pi(k)})}$ are fixed. According to the induction hypothesis for every choice of 
$a=a_{\pi(k+1)}\in G^{(i_{\pi(k+1)})}$ there exists an index $m_a$ 
(depending on the choice of $a$) such that the set 
\[X(a) =\left\{ w(x_1,\dots,x_n) \mid x_j\in G^{(i_j)}, \ x_j=a_j \textrm{ if } j\in \{\pi(1),\dots,\pi(k+1)\}\right\}\]
is contained in  $Y_{m_a}$. Of course, the set of all such $w$-values is also contained in the abelian subgroup $w({\bf i})$.

Proposition \ref{skew congruence} implies that for any $a,a'\in G^{(i_{\pi(k+1)})}$ we have 
\begin{eqnarray}
\label{X(a'a)}X(a'a)&\subseteq& X(a')X(a);\\
\label{X(bh)}X(a)&\subseteq& X(a')^{-1}X(a'a).
\end{eqnarray} 
For example, to prove the inclusion in (\ref{X(bh)}) we set 
  $\pi(k+1)=q$ and we take an element $w(a_1,\dots,a_{q-1},a,a_{q+1},\dots,a_n)\in X(a)$.
 Then by Proposition \ref{skew congruence}  we have:
\begin{eqnarray*}
& w(a_1,\ldots,a_{q-1},a'a,a_{q+1},\ldots,a_{n}) 
= \\
& w(\overline a_1,\ldots,\overline a_{q-1},a',\overline a_{{q}+1},\ldots,\overline a_{n})
w (a_1,\ldots,a_{{q}-1},a,a_{q+1},\ldots,a_{n}),
\end{eqnarray*}
where $\overline a_j$ is a conjugate of $a_j$ and moreover 
$\overline a_j=a_j$  if $i_j \leq i_q$. As $q=\pi(k+1)$, we have $i_j \leq i_q$ 
whenever $j\in \{\pi(1),\dots,\pi(k),\pi(k+1)\}$, \linebreak
so $w(\overline a_1,\ldots,\overline a_{q-1},a',\overline a_{{q}+1},\ldots,\overline a_{n})\in X(a')$.
Now (\ref{X(bh)}) follows.

For every integer $t$ we denote by $S_t$ the set of all possible $a\in G^{(i_{\pi(k+1)})}$ such that  $X(a) \subseteq Y_t$.
 The sets $S_t$ are closed sets and they cover $G^{(i_{\pi(k+1)})}$. So by  Baire category theorem   \cite[p.\ 200]{ke} there exist an index $m$, 
an element  $b \in G^{(i_{\pi(k+1)})}$ and an open normal subgroup $H\leq G^{(i_{\pi(k+1)})}$ such that for every $h\in H$ 
 we have  $X(bh)\subseteq Y_m $. 
 By (\ref{X(bh)}) we have 
$X(h)\subseteq X(b)^{-1}X(bh)$, so $X(h) \subseteq Y_m$   for every $h\in H$.   

Let $b_1,\ldots,b_l$ be a transversal of $H$ in $G^{(i_{\pi(k+1)})}$ and let $m_j=m_{b_j}$ be the indexes such that $X(b_j) \subseteq Y_{m_j}$ for $j=1,\dots,l$.
 Take an arbitrary element $g\in G^{(i_{\pi(k+1)})}$ and write $g=b_jh$ for suitable $b_j\in\{b_1,\ldots,b_l\}$ and $h\in H$; 
 then $X(g)\subseteq X(b_j)X(h)\subseteq  Y_{m_j}  Y_m \subseteq Y_{m_j+ m}.$ 
  We conclude that  for every $g\in G^{(i_{\pi(k+1)})}$ the set $X(g)$ is contained in the $\Si$-subgroup  
  $Y_{m_1+ \dots+ m_l +m}$and the inductive step is complete.

Thus, indeed for every $k=0,1,\dots,n$ and every choice of $a_{\pi(1)}\in G^{(i_{\pi(1)})}$, $a_{\pi(2)}\in G^{(i_{\pi(2)})}$, 
$\dots$, $a_{\pi(k)}\in G^{(i_{\pi(k)})}$ there exists an index $s$ such that every $w$-value $w(x_1,\dots,x_n)$, where $x_j\in G^{(i_j)}$ 
and $x_j=a_j$ whenever $j\in \{\pi(1),\dots,\pi(k)\}$, belong to $Y_s$. In the case where $k=0$ this means that $w({\bf i})\le Y_s,$ 
for an index $s$, and so $w({\bf i})$  has 
the property $\Si$.

 We can now pass to the quotient $G/w({\bf i})$: as the property $\Si$  is closed under taking  quotients, the group  
  $G/w({\bf i})$ satisfies the assumption of the lemma. 
 The induction on the number of ${\bf j}\in I$ such that $w({\bf j})\neq1$ leads us to the conclusion that 
 $w(G/w({\bf i}))$  satisfies  property $\Si$. As the property $\Si$  is  extension closed,  $w(G)$  satisfies property $\Si$.
 The proof is now complete.
\end{proof}

We will use without any further reference two more results. The first one is well-known (see for example \cite[Lemma 4.1]{S2}). 
\begin{lemma}\label{lem:delta_k}
Let $G$ be a group and  let $w$ be a multilinear commutator word on $n$ variables.  
Then each $\de_n$-value is a $w$-value.
\end{lemma}

\begin{lemma}\label{lem:torsion}
A periodic virtually soluble profinite group $G$ is locally finite and has finite exponent.  
\end{lemma}
\begin{proof} The fact that a periodic abelian profinite group has finite exponent (and therefore is locally finite) is well-known (see  for instance Exercise 10 of Chapter 2 in \cite{book:wilson}).  The general case follows easily by induction on the derived length of an open soluble normal subgroup of $G$. 
\end{proof}

In the next lemma we use same arguments as in \cite[Theorem 1.1]{AS}. 
For the reader's convenience we provide the proof.

\begin{lemma}\label{lem:reduction}
In both Theorem \ref{theoa} and Theorem \ref{theob} it is sufficient to assume that $G$ is virtually soluble. 
\end{lemma}
\begin{proof}
Assume that Theorem \ref{theoa} (resp. Theorem \ref{theob}) holds for virtually soluble groups. 
Let $\Si$ be the property of being periodic (resp. of being of finite rank).
 Let $w$ be a multilinear commutator word and $G$ a profinite group that has countably many $\Si$-subgroups $G_1, G_2, \ldots$ whose union contains all $w$-values in $G$. We wish to prove that these assumptions imply that $w(G)$ has the property $\Sigma$.

Let $k$  be the number of variables involved in $w$, so that every $\delta_k$-value is a $w$-value. For each positive integer $i$ let 
\[ S_i = \{ (x_1, \ldots, x_{2^k}) \in G \times \cdots \times G \mid 
 \delta_k(x_1, \ldots, x_{2^k}) \in G_i \}.\]

Note that the sets $S_i$ 
are closed in $ G \times \cdots \times G $
and cover the whole of $ G \times \cdots \times G $.
 By Baire category theorem at least one of
these sets contains a non-empty interior. Hence, there exist an open
subgroup $H$ of $G$,   elements 
$a_1, \ldots  , a_{2_k} \in G$  
and an integer $j$  such that
\[\delta_k(a_1H, \ldots, a_{2^k}H) \subseteq G_j.\]

Without loss of generality we can assume that the subgroup $H$ is normal in $G$. Let $K$ be the subgroup of $G$ generated by all commutators of the form $\delta_k(a_1h_1, \ldots, a_{2_k}h_{2^k}) $ where $h_i \in H$. Note that
$K \le  G_j$ and that $H$ normalizes $K$.
 Since $G_j$  is a  $\Sigma$-group, so is $K.$
 Let $D=K \cap H$.  Then $D$ is a normal  $\Sigma$-subgroup of
 $H$ and the normalizer of $D$ in $G$
has finite index. Therefore there are only finitely many conjugates of
$D$ in $G$.
Let $D =D_1, D_2 , \ldots , D_r$
be all these conjugates. All of them are normal in
$H$ and so their product
$D_1 D_2 \cdots  D_r$ is a  $\Sigma$-group. 
By passing to the quotient
$G/(D_1 D_2 \cdots  D_r)$ 
we may assume that $D = 1.$ Since
$D=K\cap H$
and $H$
has finite index in $G$, it follows that
$K$ is finite.
On the other hand, the normalizer of
$K$ has finite index in
$G$ and so the normal closure, say
$L,$ of $K$ in $G$ is also finite. We can pass to
the quotient group $G/L$ and assume that $K= 1.$ In that case we have
 \[\delta_k(a_1H, \ldots, a_{2^k}H) =1.\]
 Now by \cite[Lemma 2.2]{AS} 
the subgroup $H$
is soluble and $G$ is virtually soluble.  Our assumptions imply 
that $w(G)$ has the property $\Si$, as required.
\end{proof}

\begin{lemma}  \label{distribute} Assume that $G$ is a group and $M$ is a normal abelian subgroup of $G$. If $G/M$ is generated by $b_1M,\dots,b_sM$,  
then \[[M,G,\dots,G]=\prod [M,b_{i_1},\dots,b_{i_k}]\]
 where  $G$ appears $k$  times and the product is taken over all possible choices of $ b_{i_1}, \ldots , b_{i_{k}}\in\{ b_1, \ldots b_s \}$. 
\end{lemma} 

\begin{proof} The proof is by induction on $k$. If $k=1$, the result follows from  the 
fact that  $[x,yz]=[x,z]\,[x,y]^z$ for all $x,y,z \in G$, taking into account that $[M,b_i]^{b_j}\le  [M,b_i][M,b_j]$.     
Suppose that $k>1$ and argue by induction on $k$. Set $H=[M,G,\dots,G]$, where $G$ appears $k-1$ times. By induction \[H=\prod [M,b_{i_1},\dots,b_{i_{k-1}}],\] where the product is taken over all choices of $ b_{i_1}, \ldots,b_{i_{k-1}}\in\{ b_1, \ldots b_s \}$. Now use the case $k=1$ applied to $[H,G]$ and the fact that $[h_1h_2,g]=[h_1,g][h_2,g]$ 
for every $h_1,h_2\in H$ and $g\in G$, as $H\le M$ is abelian.
\end{proof}

\begin{lemma}\label{perfect}
 Let $G$ be a finitely generated perfect profinite group and let $w$ be a multilinear commutator word. Then
 every element of $G$ is the product of finitely many $w$-values.
\end{lemma}
\begin{proof}
The proof is by induction on the number $n$ of variables involved in $w$. If $n=1$ then the result is obvious. Assume that $n\geq 2$ and write $w=[w_1,w_2]$ where $w_1$ and $w_2$ are multilinear commutator words on less variables  than $w$. As $G$ is finitely generated, the theorem of Nikolov and Segal \cite{NS} tells us that the derived subgroup of $G$ coincides with the abstract subgroup of $G$ generated by commutators. Since $G$ is perfect, it follows that every $g\in G$ is the product of finitely many commutators, i.e.
\[g=\prod_{i=1}^n[a_i,b_i] \]
for some $a_i,b_i\in G$. By induction, each $a_i$ (resp. $b_i$) is the product of finitely
many $w_1$-values (resp. $w_2$-values). Using the well-known commutator identities $[xy,z]=[x,z]^y\,[y,z]$, $[x,yz]=[x,z]\,[x,y]^z$ we can decompose each commutator $[a_i,b_i]$ as a product of finitely many $w$-values, so the lemma follows.
\end{proof}

\begin{lemma}\label{lem:unico}
Let $k$ be a positive integer. Suppose that a profinite group $G$ contains a normal abelian subgroup $M$ such that $G/M$ is perfect and finitely generated. Then every element of $[M,G]$ is a product of finitely many $\de_k$-values, each lying in $[M,G]$.
\end{lemma} 
\begin{proof} As $G/M$ is perfect and finitely generated, by Lemma \ref{perfect} there exist finitely many $\de_{k-1}$-values, say $b_1, \ldots b_n$, such that $G=\langle b_1,\ldots,b_n\rangle M$.

Since $M$ is normal and abelian, by the Tree Subgroups Lemma
 we have $[G,M]=[G,G,M]\leq [M,G,G]$,  so $[M,G]=[M,G,G]$. It follows that  $[M,G]=[M, G, \ldots ,G]$ where 
 $G$ is taken $k$ times. By Lemma \ref{distribute}
 we have
\[[M,G]  = \prod [M, b_{i_1}, \ldots , b_{i_{k}}] \] 
 where the product is over all possible choices of $ b_{i_1}, \ldots , b_{i_{k}}$ in $\{ b_1, \ldots b_n \}$. 
Then, as $M$ is abelian, 
\[[M, b_{i_1}, \ldots , b_{i_{k}}]= \{ [m, b_{i_1}, \ldots , b_{i_{k}}] \mid m\in M \}\]
 and it is easy to see that each element $[m, b_{i_1}, \ldots , b_{i_{k}}]$ is a $\de_k$-value  lying in $[M,G]$ (cf. Lemma 2.3 in \cite{S1}).   
\end{proof}

\begin{lemma}\label{lem:baire}
Let $G$ be profinite group covered by countably many subgroups $S_1, S_2, \ldots $. Then at least one of them is open.
\end{lemma} 
\begin{proof}
By Baire category theorem, there is a subgroup $S_i$ with nonempty interior.
 This is an open subgroup.
\end{proof}

We are now in the position to prove Theorems \ref{theoa} and \ref{theob}. 

\begin{proof}[Proof of Theorem \ref{theoa}]
 Let $w$ be a multilinear commutator word and $G$ a profinite group that has countably many periodic subgroups $G_1, G_2, \ldots$ whose union contains all $w$-values in $G$. We wish to prove that $w(G)$ is locally finite.
 By Lemma \ref{lem:reduction}, it is sufficient to prove this in the case where $G$ is virtually soluble. Thus, we assume that $G$ is virtually soluble.

 By Lemma 2.4 of \cite{AS2} the derived series of $G$ has only finitely many terms, that is, $G^{(i)}=G^{(i+1)}$ for some positive integer $i$. It is sufficient to prove that $G^{(i)}$ is locally finite, because then factoring it out we get the case where $G$ is soluble and the result is immediate from Lemma \ref{lem:sigma}. So by replacing $G$ with $G^{(i)}$ we can assume that $G$ is a perfect virtually soluble profinite group. We will show that $G$ is locally finite.

Among all open normal soluble subgroups of $G$ we choose $N$  with minimal derived length, say $d$. 
The proof is by induction on $d$. If $d=0$, then $G$ is finite and there is nothing to prove.
So let $d>0$.
Let $M$ be the last nontrivial
term of the derived series of $N$. By induction $\ol G=G/M$ is locally finite.
 For short, if $X$ is a
 subgroup of $G$  then $\ol X$  will denote its image in $\ol G$.

Let $T$ be a subgroup of $G$ with the property that $\ol T$ is a minimal finite subgroup of 
$\ol G$ such that $\ol G=\ol N\ol T$. Note that such $T$ exists because $\ol G/\ol N$ is finite and
$\ol G$ is locally finite.
Since $G$ is perfect it follows that $NT=NT'$, hence, by minimality of   $T$, we deduce that $\ol T$ is perfect.

Also, $\ol G$ is the normal closure  of $\ol T$ in $\ol G$ because $\ol G/\ol T^{\ol G}$ is a perfect group which is isomorphic to a quotient of the soluble group $\ol N$.

By {a profinite version of} Lemma \ref{lem:delta_k}, there exists an integer $k$ such that every $\de_k$-value is a $w$-value. We apply Lemma \ref{lem:unico} and deduce that every element of $[M,T]$ is a product of finitely many $\delta_k$-values,
each lying in $[M,T]$. For every integer $i$ define  
\[S_i = \prod_{j=1}^i (G_j \cap [M,T]).\]
We see that countably many periodic subgroups $S_i$ cover $[M,T]$. By  Lemma \ref{lem:baire} at least one of them is open in $[M,T]$. As the class of periodic groups is closed under taking extensions, we deduce that $[M,T]$ is an abelian periodic group. Hence by Lemma \ref{lem:torsion} we obtain that  $[M,T]$ has finite exponent.

It follows that the normal closure of $[M,T]$ in $G$ has finite exponent. Factoring it out, we may assume that $M$ centralizes $T$. So $M$ centralizes the normal closure of $T$ and as $G=T^GM$ it follows that $M$ is contained in the center $Z(G)$ of $G$. By induction $G/M$ is locally finite and so by Lemma \ref{lem:torsion} it has finite exponent. Hence, $G/Z(G)$ has finite exponent as well. A theorem of Mann \cite{M} states that if $B$ is a finite group such that $B/Z(B)$ has exponent $e$, then the exponent of $B'$ is bounded by a function depending on $e$ only. Applying a profinite version of this theorem we deduce that the exponent of $G'=G$ is finite. In particular $G$ is locally finite, as  required. 
\end{proof}

\begin{proof}[Proof of Theorem \ref{theob}]
 Let $w$ be an multilinear commutator word and $G$ a profinite group that has countably many finite rank subgroups $G_1, G_2, \ldots$ whose union contains all $w$-values in $G$. We wish to prove that $w(G)$ has finite rank. As in the proof of Theorem \ref{theoa}, without loss of generality we assume that $G$ is perfect and virtually soluble. We will show that $G$ has finite rank.

Among all open normal soluble subgroups of $G$ we choose $N$ with minimal derived length $d$. The proof is by induction on $d$. If $d=0$, the result is obvious. So we assume that $d \ge 1$ and let $M$ be the last nontrivial term of the derived series of $N$. By induction $G/M$ has finite rank. In particular, $G/M$ is finitely generated. Let $k$ be the  number of variables involved in $w$. By Lemma \ref{lem:unico}, every element of $[M,G]$ is a product of finitely many $\de_k$-values, each lying in $[M,G]$.
For every integer $i$ define  $S_i = \prod_{j=1}^i (G_j \cap [M,G])$ and note that each subgroup $S_i$ has finite rank. 
Since by Lemma \ref{lem:delta_k}  every $\delta_k$-value is a $w$-value,  $[M,G]$  is covered by the subgroups $S_1, S_2, \ldots $
 and by  Lemma \ref{lem:baire} one of them is open, hence $[M,G]$ has finite rank. 

Passing to the quotient $G/[M,G]$, we can assume that $[M,G]=1$. As $G/M$ has finite rank, we deduce that $G/Z(G)$ has finite rank. 
 By a result of Mann and Lubotzky \cite{powerful} (see also \cite{rank}), this implies that $G'$ has finite rank. Since $G$ is perfect, the proof is complete. 
\end{proof}

We will now give the proof of Theorem \ref{theoc}. First, we need a modification of Lemma \ref{lem:sigma}.
\begin{lemma}\label{prop:finite-order} Suppose that $w$ is a multilinear commutator word. Let $G$ be a soluble profinite group in which every $w$-value has finite order. Then $w(G)$ has finite exponent.
\end{lemma}
\begin{proof}
We follow the same arguments as in the proof of Lemma \ref{lem:sigma}. 
Choose ${\bf i} \in I$ such that $w({\bf i}) \neq 1$  while $w({\bf j})=1$ 
whenever ${\bf i}<{\bf j}$. By Corollary \ref{cor:ab} it follows that  $w({\bf i})$ is abelian. We define $Y_s$ to be the subgroup generated by the $w$-values in $w({\bf i})$ whose order divides $s$: 
\[   Y_s =\langle w(x_1, \ldots, x_n) \in w({\bf i}) \mid w(x_1, \ldots, x_n)^s=1 \rangle. \] 
Since  $w({\bf i})$ is abelian, every element in $Y_s$ has order dividing $s$, hence $Y_s$ has finite exponent. 
Then we  argue as in  Lemma \ref{lem:sigma}, where $\Si$ is the property of having finite exponent, to deduce  that  $w({\bf i})$ has finite exponent. 

Remark that the quotient $G/w({\bf i})$ satisfies the hypothesis of the lemma. We can therefore pass to the quotient $G/w({\bf i})$. By induction on the number of ${\bf j}\in I$ such that $w({\bf j})\neq1$ 
 we obtain that  $w(G/w({\bf i}))$ has finite exponent. Thus  $w(G)$  has finite exponent and the proof is now complete.
\end{proof}

\begin{proof}[Proof of Theorem \ref{theoc}]
By Lemma \ref{prop:finite-order} the result holds
when $G$ is soluble. Now we have to prove it when  $G$ is virtually soluble.
This situation was in fact considered in the course of the proof of  Theorem \ref{theoa}. It is easy to see that only obvious modifications of the proof of Theorem \ref{theoa} are required now. Thus, the result follows.
\end{proof}

\noindent{\bf{Acknowledgments}}
We thank Gustavo Fern\'andez-Alcober for his contribution to the proofs of some results given in Section 2.

We also thank the funding agencies MIUR for partial support of the first two authors and CAPES and CNPq for support of the third author.

\end{document}